\documentclass[12pt]{amsart}
\usepackage{amsmath,amsthm}
\numberwithin{equation}{section}
\newtheorem{thm}[equation]{Theorem}

\newtheorem{lm}[equation]{Lemma}
\newtheorem{prp}[equation]{Proposition}

\theoremstyle{definition}
\newtheorem{df}[equation]{Definition}

\theoremstyle{remark}
\newtheorem{rem}[equation]{Remark}
\newcommand{\sprf}{\noindent{\it Proof.}}
\newcommand{\sqed}{\hfill\rule{1.3mm}{3mm}\medskip}

\newcounter{stareq}
\def\thestareq{\fnsymbol{stareq}}

\DeclareMathOperator{\miff}{\Longleftrightarrow}

\DeclareMathOperator{\BR}{\mathbb{R}} 

\DeclareMathOperator{\inter}{\cap}  

\newcommand{\bd}{\begin{description}}
\newcommand{\ed}{\end{description}}
\newcommand{\ol}{\overline}

\newcommand{\vep}{\varepsilon}

\begin{document}

\date{\today}

\title{Singularities and nonhyperbolic manifolds \\ do not coincide}

\author{N\'andor Sim\'anyi}

\address{The University of Alabama at Birmingham \\
Department of Mathematics \\
1300 University Blvd., Suite 452 \\
Birmingham, AL 35294 U.S.A.}

\email{simanyi@uab.edu}



\begin{abstract}
We consider the billiard flow of elastically colliding hard balls on
the flat $\nu$-torus ($\nu\ge 2$), and prove that no singularity
manifold can even locally coincide with a manifold describing future
non-hyperbolicity of the trajectories. As a corollary, we obtain the
ergodicity (actually the Bernoulli mixing property) of all such
systems, i.e. the verification of the Boltzmann-Sinai Ergodic
Hypothesis.
\end{abstract}

\maketitle

\section{Introduction} \label{introduction}

In this paper we prove the Boltzmann--Sinai Ergodic Hypothesis for 
hard ball systems on the $\nu$-torus $\mathbb{R}^\nu/\mathbb{Z}^\nu$ ($\nu\ge 2$) without 
any assumed hypothesis or exceptional model.

This introduction is, to a large extent, an edited version of some
paragraphs of the introductory sections \S1 and \S2 of my paper
\cite{Sim(2009)}. For a more detailed introduction into the topic of
hard ball systems, please see these two sections of \cite{Sim(2009)}.

In a loose form, as attributed to L. Boltzmann back in the 1880's, the
Boltzmann hypothesis asserts that gases of hard balls are ergodic. In
a precise form, which is due to Ya. G. Sinai in \cite{Sin(1963)}, it
states that the gas of $N\ge 2$ identical hard balls (of "not too big"
radius) on a torus $\mathbb{T}^\nu=\mathbb{R}^\nu/\mathbb{Z}^\nu$, $\nu \ge 2$,
(a $\nu$-dimensional box with periodic boundary conditions) is ergodic, provided that certain
necessary reductions have been made. The latter means that one fixes
the total energy, sets the total momentum to zero, and restricts the
center of mass to a certain discrete lattice within the torus. The
assumption of a not too big radius is necessary to have the interior
of the arising configuration space connected.

Sinai himself pioneered rigorous mathematical studies of hard ball
gases by proving the hyperbolicity and ergodicity for the case $N=2$
and $\nu=2$ in his seminal paper \cite{Sin(1970)}, where he laid down
the foundations of the modern theory of chaotic billiards. The proofs
there were further polished and clarified in \cite{B-S(1973)}. Then
Chernov and Sinai extended these results to ($N=2$, $\nu\ge 2$), as
well as proved a general theorem on ``local'' ergodicity applicable to
systems of $N>2$ balls \cite{S-Ch(1987)}; the latter became
instrumental in the subsequent studies. The case $N>2$ is
substantially more difficult than that of $N=2$ because, while the
system of two balls reduces to a billiard with strictly convex
(spherical) boundary, which guarantees strong hyperbolicity, the gases
of $N>2$ balls reduce to billiards with convex, but not strictly
convex, boundary (the latter is a finite union of cylinders) -- and
those are characterized by a weak hyperbolicity.

Further development has been due mostly to A. Kr\'amli,
D. Sz\'asz, and the present author. We proved the hyperbolicity and
ergodicity for $N=3$ balls in any dimension \cite{K-S-Sz(1991)} by 
exploiting the ``local'' ergodic theorem of Chernov and Sinai 
\cite{S-Ch(1987)}, and carefully analyzing all
possible degeneracies in the dynamics to obtain ``global''
ergodicity. We extended our results to $N=4$ balls in
dimension $\nu \ge 3$ next year \cite{K-S-Sz(1992)}, and then I proved the
ergodicity whenever $N\le\nu$ in \cite{Sim(1992)-I} and \cite{Sim(1992)-II}.
At that point the existing methods could no longer handle any new cases, because
the analysis of the degeneracies became overly complicated. It was
clear that further progress should involve novel ideas.

A big step forward was made by D. Sz\'asz and myself, when we used the methods of
algebraic geometry in \cite{S-Sz(1999)}. We assumed that the balls had arbitrary
masses $m_1,\dots,m_N$ (but the same radius $r$). By taking the limit
$m_N\to 0$, we were able to reduce the dynamics of $N$ balls to the motion of
$N-1$ balls, thus utilizing a natural induction on $N$. Then algebro-geometric
methods allowed us to effectively analyze all possible degeneracies, but only
for typical (generic) $(N+1)$-tuples of ``external'' parameters
$(m_1,\dots,m_N,r)$; the latter needed to avoid some exceptional submanifolds
of codimension one, which remained unknown. This approach led to a proof of
full hyperbolicity (but not yet ergodicity) for all $N\ge2$ and $\nu\ge2$, and
for generic $(m_1,\dots,m_N,r)$, see \cite{S-Sz(1999)}. Later I
simplified the arguments and made them more ``dynamical'', which allowed me to
obtain full hyperbolicity for hard balls with any set of external geometric
parameters $(m_1,\dots,m_N,r)$ \cite{Sim(2002)}. The reason why the masses $m_i$
are considered \emph{geometric parameters} is that they determine the relevant
Riemannian metric
\[
||dq||^2=\sum_{i=1}^N m_i||dq_i||^2
\]
of the system. Thus, the complete hyperbolicity has been fully 
established for all systems of hard balls on tori.

To upgrade the complete hyperbolicity to ergodicity, one needs to refine the
analysis of the mentioned degeneracies. For hyperbolicity, it was enough
that the degeneracies made a subset of codimension $\ge 1$ in the phase space.
For ergodicity, one has to show that its codimension is $\ge 2$, or find
some other ways to prove that the (possibly) arising codimension-one manifolds
of non-sufficiency are not capable of separating distinct ergodic components.
In the paper \cite{Sim(2003)} I
took the first step in the direction of proving that the codimension of
exceptional manifolds is at least two: I proved that the systems of $N \ge 2$
balls on a $2$-dimensional torus are ergodic for typical (generic)
$(N+1)$-tuples of external parameters $(m_1,\dots,m_N,r)$. The proof again
involves some algebro-geometric techniques, thus the result is restricted to
generic parameters $(m_1,\dots,m_N;\,r)$.  But there was a good reason to
believe that systems in $\nu\ge 3$ dimensions would be somewhat easier to
handle, at least that was indeed the case in earlier studies.

As the next step, in the paper \cite{Sim(2004)} I was able to further improve the
algebro-geometric methods of \cite{S-Sz(1999)}, and proved that for any $N\ge 2$,
$\nu\ge 2$, and for almost every selection $(m_1,\dots,m_N;\,r)$ of the external
geometric parameters the corresponding system of $N$ hard balls on 
$\mathbb{T}^\nu$ is (completely hyperbolic and) ergodic.

Finally, in the paper \cite{Sim(2009)} I managed to prove the Boltzmann-Sinai 
Ergodic Hypothesis in full generality (i. e. without exceptional models), by assuming 
that the so called Chernov-Sinai Ansatz is true for these models.

\begin{rem}
The Chernov-Sinai Ansatz states that for almost every singular phase point $x\in\mathcal{SR}^+_0$
(with respect to the hypersurface measure of $\mathcal{SR}^+_0$) the forward orbit $S^{(0,\infty)}x$
is sufficient (geometrically hyperbolic). This is the utmost important global geometric hypothesis of
the Theorem on Local Ergodicity of \cite{S-Ch(1987)}, see also Condition 3.1 in \cite{K-S-Sz(1990)}.
\end{rem}

The only missing piece of the whole puzzle is to prove that no
open piece of a singularity manifold can precisely coincide with a
codimension-one manifold desribing the trajectories with a
non-sufficient forward orbit segment corresponding to a fixed symbolic
collision sequence. This is exactly what we prove in our Theorem below.

\section{Formulation and Proof of the Theorem} \label{main-section}

Let $U_0\subset \mathbf{M}\setminus\partial\mathbf{M}$ be an open ball, $T>0$, and assume that

\medskip

(a) $S^T(U_0)\cap \partial\mathbf{M}=\emptyset$,

\medskip

(b) $S^T$ is smooth on $U_0$.

\medskip

Next we assume that there is a \emph{codimension-one}, smooth
submanifold $J\subset U_0$ with the property that for every $x\in U_0$
the trajectory segment $S^{[0,T]}x$ is geometrically hyperbolic
(sufficient) if and only if $x\not\in J$. ($J$ is a so called
non-hyperbolicity or degeneracy manifold.) Denote the common symbolic
collision sequence of the orbits $S^{[0,T]}x$ ($x\in U_0$) by
$\Sigma=(e_1,e_2,\dots,e_n)$, listed in the increasing time
order, and let the corresponding advances be $\alpha_i=\alpha(e_i)$,
$i=1,2,\dots,n$. Let $t_i=t(e_i)$ be the time of the $i$-th collision,
$0<t_1<t_2<\dots<t_n<T$.

Finally we assume that for every phase point $x\in U_0$ the first reflection 
$S^{\tau(x)}x$ \emph{in the past} on the orbit of $x$ is a singular reflection
(i. e. $S^{\tau(x)}x \in \mathcal{SR}^+_0$) if and only if $x$ belongs to a codimension-one,
smooth submanifold $K$ of $U_0$. For the definition of the manifold of singular
reflections $\mathcal{SR}^+_0$ see, for instance, the end of \S1 in \cite{Sim(2009)}.

\begin{thm}
Using all the assumtions and notations above, the submanifolds $J$ and $K$ of $U_0$
do not coincide.
\end{thm}

The rest of this section is devoted to the proof of this theorem. It
will be a proof by contradiction, so from now on we assume that $J=K$,
and the proof will be subdivided into several lemmas and propositions.

First of all, we assume that the center $x_0$ of the open ball $U_0$
belongs to the exceptional set $J$. During the indirect proof of the
theorem, smaller and smaller open balls $U_0$ will be selected to
guarantee a regular (smooth and homogeneous) behavior. We note that this can be done,
thanks to the algebraic nature of the dynamics.

Observe that the sufficiency of the orbit segments $S^{[0,T]}x$  ($x\in U_0\setminus J$) immediately implies that the collision
graph $\mathcal{G}=\left(\{1,2,\dots,N\},\, \{e_1,e_2,\dots,e_n\}\right)$ is connected on the vertex set
$\mathcal{V}=\{1,2,\dots,N\}$. Therefore, according to Lemma 2.13 of [Sim(1992)-II], the linear map 
\[
\Phi:\; \mathcal{N}_0\left(S^{[0,T]} x\right)\to \mathbb{R}^n
\]
defined by $\left(\Phi(w)\right)_i=\alpha_i(w)$ ($i=1,2,\dots,n$) is a
linear embedding for every $x\in U_0$. Here $\mathcal{N}_0(S^{[0,T]}x)$ denotes the neutral linear space of the
trajectory segment $S^{[0,T]}x$, see Definition 2.5 of \cite{Sim(2009)}. The image $\Phi\left(\mathcal{N}_0(S^{[0,T]}x)\right)$
will be denoted by $\ol{\mathcal{N}}_0(S^{[0,T]}x)$. The sufficiency (geometric hyperbolicity)
of a trajectory segment $S^{[0,T]}x$ means that the dimension of the neutral linear space $\mathcal{N}_0(S^{[0,T]}x)$
takes the minimum possible value $1$, see Definition 2.7 in \cite{Sim(2009)}.
Moreover, let $1=k(1)<k(2)<\dots<k(N-1)<n$ be the uniquely defined
indices with the property that for every $l$ ($1\le l\le N-1$) the
collision graph $\left(\mathcal{V},\, \{e_1,e_2,\dots,e_{k(l)}\}\right)$
has exactly $N-l$ connected components, whereas the number of
components of $\left(\mathcal{V},\, \{e_1,e_2,\dots,e_{k(l)-1}\}\right)$ is
$N-l+1$.

We shall call the edges (collisions) $e_{k(1)},\dots,e_{k(N-1)}$ {\it essential}.

For every {\it non-essential edge} $e_m=\{i(m),j(m)\}$ ($1\le i(m)<j(m)\le N$) we express the relative displacement 
\[
\Delta q^-_{i(m)}(t_m)-\Delta q^-_{j(m)}(t_m) =\alpha_m\left[v^-_{i(m)}(t_m)-v^-_{j(m)}(t_m)\right]
\]
as a linear combination of relative velocities of earlier collisions $e_1, e_2,\dots,e_{m-1}$
(with coefficients made up from some masses and advances) precisely as described by the CPF,
see Proposition 2.19 in \cite{S-Sz(1999)}:

\begin{equation}\label{CPF}
\alpha_m\left[v^-_{i(m)}(t_m)-v^-_{j(m)}(t_m)\right]=\sum_{k=1}^{m-1} \alpha_k \Gamma_k^{(m)}
\end{equation}
($1\le m\le n$, $e_m$ is not essential), where each $\Gamma_k^{(m)}$ is a linear combination of the relative velocities
$v^-_{i(k)}-v^-_{j(k)}$ and $v^+_{i(k)}-v^+_{j(k)}$, and the coefficients in these linear combinations are fractional linear expressions
of the masses $m_{i(k)}$ and $m_{j(k)}$, see the CPF as Proposition 2.19 in \cite{S-Sz(1999)}. 
We observe that the solution set of the system of all equations (\ref{CPF})
(taken for all $m$ with a non-essential edge $e_m$) is precisely the linear space 
$\Phi\left(\mathcal{N}_0(S^{[0,T]} x)\right)=\ol{\mathcal{N}}_0(S^{[0,T]} x)$, having the same dimension as the neutral space 
$\mathcal{N}_0(S^{[0,T]} x)$, $x\in U_0$.

As follows, we are presenting an indirect proof (a proof by contradiction) by assuming that the nonhyperbolicity manifold $J$ coincides 
with a past singularity so that no collision takes place between the mentioned singularity and $J$. (Otherwise those collisions
between the singularity and $J$ could be added to the symbolic sequence $\Sigma=(e_1,e_2,\dots,e_n))$ as an initial segment.)

Throughout the proof we shall assume that the masses of the
elastically interacting balls are equal: $m_1=m_2=\dots=m_N$. As a
matter of fact, this assumption is not a serious restriction of
generality: it is merely a technical-notational assumption, and the
reader can easily re-write the present proof to cover the general case
of arbitrary masses. We denote by $d=\nu(N-1)$ the dimension of the
configuration space $\mathbf{Q}$.

\medskip

Following the ideas and notations of \S3 of \cite{S-Sz(2000)}, we introduce the following notions and notations.

\medskip

With every collision $e_k=(i(k),\, j(k))$ ($1\le k\le n$, $1\le i(k)<j(k)\le N$) we associate the real projective space
$\mathcal{P}\cong \mathbb{RP}(\nu-1)$ of all orthogonal reflections of the common tangent space 
\begin{equation} \label{tangent-space}
\mathcal{Z}=\mathcal{T}\mathbf{Q}=\mathcal{T}_q\mathbf{Q} \\
=\left\{(\delta q_1,\dots,\delta q_N)\in(\mathbb{R}^\nu)^N \big|\;
\sum_{i=1}^N \delta q_i=0\right\}\cong\mathbb{R}^d
\end{equation}
across all possible tangent hyperlanes $H$ of the cylinder $C_{e_k}$ corresponding to the collision $e_k$. In this way we obtain a map
\begin{equation} \label{definition-Phi}
\Phi:\; S^{d-1}\times \prod_{k=1}^n \mathcal{P}_k \to S^{d-1}
\end{equation}
which assignes to every $(n+1)$-tuple 
\[
(V_0;\, g_1,g_2,\dots,g_n)\in S^{d-1}\times\prod_{k=1}^n \mathcal{P}_k
\]
the image velocity $V_n=V_0g_1g_2\dots g_n$ of $V_0$ under the composite action $g_1g_2\dots g_n$. (Here, by convention, the composition
is carried out from the left to the right, and $S^{d-1}$ denotes the unit sphere of $\mathcal{Z}$ in \ref{tangent-space}.)
The space $M_n=S^{d-1}\times\prod_{k=1}^n \mathcal{P}_k$ is called the phase space of the virtual velocity process 
$(V_0,V_1,\dots,V_n)$, where $V_k=V_0g_1g_2\dots g_k$. Clearly, the velocity process $(V_0,V_1,\dots,V_n)$ uniquely determines the
sequence of reflections $g_1,g_2,\dots,g_n$. For any $x\in M_n$ or $x\in U_0$ we denote the velocity $V_k$ by $V_k(x)$. Similarly,
$v^+_{i(k)}-v^+_{j(k)}$ denotes the relative velocity of the colliding particles $i(k)$ and $j(k)$ right after the collision
$e_k=(i(k),\, j(k))$ ($1\le i(k)<j(k)\le N$), and the definition of the pre-collision relative velocity $v^-_{i(k)}-v^-_{j(k)}$ is 
analogous, $k=1,2,\dots,n$. Thus we get a natural projection
\begin{equation} \label{projection}
\Pi:\; U_0\to M_n
\end{equation}
by taking $\Pi(x)=(V_0(x);\, g_1(x),\dots,g_n(x))$ for $x=(q(x),\, v(x))\in U_0$, where $V_0(x)=v(x)$.

What is coming up is a local analysis in a small, open ball neighborhood $B_0\subset M_n$ of the base point
$(V_0(x_0);\, g_1(x_0),\dots,g_n(x_0))$. We begin with a useful definition.

\begin{df} \label{projection-definition}
The projections $R_k:\; \mathcal{Z}\to \mathbb{R}^\nu$ ($k=1,2,\dots,n$) are defined by the equation
\[
R_k(\delta q)=\delta q_{i(k)}-\delta q_{j(k)}
\]
for $\delta q\in \mathcal{Z}$, where $\mathcal{Z}$ is the tangent space of $\mathbf{Q}$ in (\ref{tangent-space}).
\end{df}

The Connecting Path Formula (\ref{CPF}) together with the results of \cite{S-Sz(2000)} and \cite{Sim(2002)} yield the following
results.

\begin{prp} \label{DDP}
For any integer $m$, $2\le m\le n$, the neutral space 
\[
\mathcal{N}_{t_1+0}(V_0;\, g_1,g_2,\dots,g_{m})=\mathcal{N}_{1}(V_0;\, g_1,g_2,\dots,g_{m})
\]
is determined by the \emph{directions} of all relative velocities $v^-_{i(l)}-v^-_{j(l)}$,
$v^+_{i(l)}-v^+_{j(l)}$ ($2\le l\le m-1$), and by the directions of $v^+_{i(1)}-v^+_{j(1)}$ and
$v^-_{i(m)}-v^-_{j(m)}$. \emph{This property will be called the Direction Determination Principle, or DDP.}
As a consequence, all the neutral spaces 
\[
\mathcal{N}_{k}=\mathcal{N}_{k}(V_0;\, g_1,g_2,\dots,g_{m})=\mathcal{N}_{t_k+0}(V_0;\, g_1,g_2,\dots,g_{m})
\] 
($0\le k\le m$) are determined by the relative velocities listed above and by $v^-_{i(1)}-v^-_{j(1)}$,
$v^+_{i(m)}-v^+_{j(m)}$. We note that the neutral spaces $\mathcal{N}_{k}$ are connected to each other via
the equations $\mathcal{N}_{l}=\mathcal{N}_{k}\cdot g_{k+1}\cdot\dots\cdot g_l$ for $k<l$, and
the reflection $g_s$ is (locally) determined by the directions of the relative velocities
$v^-_{i(s)}-v^-_{j(s)}$ and $v^+_{i(s)}-v^+_{j(s)}$, $1\le s\le m$.
\end{prp}

\begin{proof}
Observe that for any tangent vector $\delta q=(\delta q_1,\dots,\delta q_N)\in\mathcal{Z}$ the relation 
$\delta q\in \mathcal{N}_1(V_0;\, g_1,\dots,g_{m})$ holds true if and only if for every $k$, $2\le k\le m$, the vector
$R_k(\delta q\cdot g_2\cdot g_3\cdot\dots\cdot g_{k-1})$ is parallel to the relative velocity vector 
$v^-_{i(k)}-v^-_{j(k)}$, and $R_1(\delta q)$ is parallel to $v^+_{i(1)}-v^+_{j(1)}$.
\end{proof}

\begin{prp}\label{variations}
Use the notions and notations of the previous proposition, except that here we allow the values $0$ and $1$ for the number of collisions
$m$. We claim that for fixed \emph{directions} of all relative velocities $v^-_{i(k)}-v^-_{j(k)}$ and $v^+_{i(k)}-v^+_{j(k)}$ ($k=1,2,\dots,m$)
and for a given reference time $t_s+0$ ($0\le s\le m$, $t_0=0$) all possible space and velocity variations $\delta q$ and $\delta v$ are
precisely the elements of the neutral space $\mathcal{N}_{s}(V_0;\, g_1,g_2,\dots,g_{m})$.
\end{prp}

\begin{proof}
Induction on $m$. The statement is obviously true for $m=0$, since in this case $\mathcal{N}_{0}(V_0;\, g_1,g_2,\dots,g_{m})=\mathcal{Z}$,
the tangent space of the configuration space.

Assume now that $m\ge1$ and the claim is true for all smaller numbers of collisions. Clearly it is enough to prove the proposition 
for the case $s=m-1$. The fixed directions of $v^-_{i(k)}-v^-_{j(k)}$ and $v^+_{i(k)}-v^+_{j(k)}$ for $k=1,2,\dots,m-1$ mean that the possible
values of either $\delta q$ or $\delta v$ are precisely the elements of the neutral space
$\mathcal{N}_{m-1}(V_0;\, g_1,g_2,\dots,g_{m-1})$. If, in addition, we also fix the direction of $v^-_{i(m)}-v^-_{j(m)}$, then this leaves for us 
the space $\mathcal{N}_{m-1}(V_0;\, g_1,g_2,\dots,g_{m})$ as the set of all available values for $\delta v$. Furthermore, by also fixing the
direction of $v^+_{i(m)}-v^+_{j(m)}$ (i. e. also fixing the reflection $g_m$) restricts the space of available values for $\delta q$ to the neutral
space $\mathcal{N}_{m-1}(V_0;\, g_1,g_2,\dots,g_{m})$.
\end{proof}

\begin{prp}\label{typ-dim}
For every $m$, $1\le m\le n$, the generic ($\miff$minimal) dimension (both in measure-theoretical and topological senses) 
of the neutral spaces 
\[
\mathcal{N}_0(V_0;\, g_1,\dots,g_{m})
\]
on the phase space $M_{m}$ is equal to the generic ($\miff$minimal) value of
\[
\text{dim}\mathcal{N}_0(V_0(x);\, g_1(x),g_2(x),\dots,g_{m}(x))
\]
for all $x\in U_0$. (Key Lemma 3.19 in \cite{Sim(2002)}.)
\end{prp}

The value of this typical dimension will be denoted by $\Delta(e_1,e_2,\dots,e_{m})$. Plainly, it only depends on the symbolic
sequence $(e_1,e_2,\dots,e_{m})$. 

The value of $\text{dim}\mathcal{N}_0(V_0(x);\, g_1(x),\dots,g_{m}(x))$ for typical $x\in J$ (either in measure-theoretical or 
in topological sense) will be denoted by $\Delta_J(e_1,e_2,\dots,e_{m})$. By selecting the open balls $B_0$ and $U_0$ 
($B_0\subset M_n$, $U_0\subset \mathbf{M}$, $U_0=\Pi^{-1}(B_0)$) small enough we may (and shall) assume that for every integer $m$,
$1\le m\le n$, 
\begin{equation}\label{typ-dim1}
\text{dim}\mathcal{N}_0\left(V_0(y);\, g_1(y),\dots,g_{m}(y)\right)
=\Delta(e_1,e_2,\dots,e_{m}) \quad \forall y\in B_0\setminus\tilde{J},
\end{equation}

\begin{equation}\label{typ-dim2}
\text{dim}\mathcal{N}_0\left(V_0(y);\, g_1(y),\dots,g_{m}(y)\right)
=\Delta_J(e_1,e_2,\dots,e_{m}) \quad \forall y \in \tilde{J},
\end{equation}
where $\tilde{J}\subset B_0$ is an analytic submanifold of $B_0$ with $J=\Pi^{-1}(\tilde{J})$.

\begin{prp}\label{key-lemma-corollary}
(A corollary of the proof of Key Lemma 3.19 of \cite{Sim(2002)}.) Let $1\le m\le n$, and $\mathcal{N}^* \subset\mathcal{Z}$
be a given subspace with $\mathcal{N}^*\cap\left\{V_m(x)\big|\; x\in U_0\right\}\ne\emptyset$.
We claim that the typical (i. e. minimal) value of
\[
\text{dim}\left[\mathcal{N}^*\cap\mathcal{N}_m(V_m;\, g_{m+1},g_{m+2},\dots,g_n)\right]
\]
for $V_m\in\mathcal{N}^*$ and $g_k\in\mathcal{P}_k$ ($m+1\le k\le n$) is equal to the typical (i. e. minimal) value of
\[
\text{dim}\left[\mathcal{N}^*\cap\mathcal{N}_m(V_m(x);\, g_{m+1}(x),g_{m+2}(x),\dots,g_n(x))\right]
\]
for $x\in U_0$ with $V_m(x)\in\mathcal{N}^*$.
\end{prp}

\begin{proof}
The proof of this statement can be obtained from the proof of Key Lemma 3.19 of \cite{Sim(2002)}, hence it is omitted.
\end{proof}

Now it is time to bring up the definition of the ``critical index'' $n_0$.

\begin{df}\label{n-zero}
The ``critical index'' $n_0$ is the unique positive integer $n_0$, $1\le n_0\le n$, with the property that for any $x\in U_0$
\begin{enumerate}
\item[(i)] the directions of the relative velocities $v_{i(k)}^-(x)-v_{j(k)}^-(x)$, $v_{i(k)}^+(x)-v_{j(k)}^+(x)$, $k=1,2,\dots,n_0$,
determine in $M_n$ if $\Pi(x)\in\tilde{J}$, whereas
\item[(ii)] the directions of the relative velocities $v_{i(k)}^-(x)-v_{j(k)}^-(x)$, $v_{i(k)}^+(x)-v_{j(k)}^+(x)$, $k=1,2,\dots,n_0-1$,
do not determine yet in $M_n$ if $\Pi(x)\in\tilde{J}$.
\end{enumerate}
The precise meaning of the notions above is the following: The manifolds $\mathcal{W}_{n_0}=\mathcal{W}_{n_0}(x)\subset U_0$ that are defined by
fixing the directions of all the relative velocities listed in (i) (which form a smooth foliation of the local neighborhood $U_0$ if $U_0$ is chosen
small enough) are either subsets of $J$ or they are disjoint from it, whereas the manifolds $\mathcal{W}_{n_0-1}=\mathcal{W}_{n_0-1}(x)$ that are defined by
fixing the directions of all the relative velocities listed in (ii) (which also form a smooth foliation of the local neighborhood $U_0$ for small enough $U_0$)
are transversal to $J$.
\end{df}

Apply Proposition \ref{key-lemma-corollary} to $m=n_0$, 
\[
\mathcal{N}^*=\mathcal{N}_m\left(V_0(x);\, g_1(x),g_2(x),\dots,g_{n_0}(x)\right)
\]
($x\in U_0$) to realize that the directions of the relative velocities listed above in (i) also determine if the phase point 
$(V_0;\, g_1,g_2,\dots,g_n)\in M_n$ belongs to $\tilde{J}$ or not. In the free velocity process
$(V_0;\, g_1,g_2,\dots,g_n)\in M_n$ there is absolutely no constraint on the velocities, other than that each $g_k$ is an
orthogonal reflection across a hyperplane determined by $e_k=(i(k),\, j(k))$. Because of this, the only way that the relative
velocities listed above in (i) determine the status of $(V_0;\, g_1,\dots,g_n)\in\tilde{J}$ is that
a minor $\mathcal{M}$ (determinant of a square submatrix) of the system (\ref{CPF}) with maximum column index $n_0$ vanishes.
Observe that the $n_0$-th column of the system of CPFs (\ref{CPF}),
i.e. the coefficients of the unknown $\alpha_{n_0}$ in (\ref{CPF}), depend on the pair of velocities
\[
r(x)=\left(v^-_{i(n_0)}(x)-v^-_{j(n_0)}(x),\, v^+_{i(n_0)}(x)-v^+_{j(n_0)}(x)\right)
\]
linearly (they are certain linear combinations of some coordinates of the two components of
$r(x)$), hence the minor $\mathcal{M}$ also depends linearly on $r(x)$, and $(V_0;\, g_1,\dots,g_n)\in\tilde{J}$ 
means that the solution set of (\ref{CPF}) is atypically big. Using these two observations and the Direction Determination 
Principle (DDP) of Proposition \ref{DDP} we obtain a useful description of the membership relation $x\in J$ as follows.

\begin{prp}\label{belongs-to-hyperplane}
For any $x\in U_0$ the relation $x\in J$ holds true if and only if the pair of relative
velocities
\begin{equation} \label{def-r}
r(x):=\left(v^-_{i(n_0)}(x)-v^-_{j(n_0)}(x),\, v^+_{i(n_0)}(x)-v^+_{j(n_0)}(x)
\right) \in \mathbb{R}^\nu\times\mathbb{R}^\nu=\mathbb{R}^{2\nu}
\end{equation}
belongs to a hyperplane $H(x)\subset\mathbb{R}^{2\nu}$ depending analytically on the directions
\[
\text{dir}(v^-_{i(k)}(x)-v^-_{j(k)}(x)), \quad
\text{dir}(v^+_{i(k)}(x)-v^+_{j(k)}(x))
\]
of the indicated relative velocities for $k=1,2,\dots,n_0-1$.
\end{prp}

\medskip

In order to make the mechanism discussed in Proposition \ref{belongs-to-hyperplane} more transparent, below we provide
the reader with a brief analysis of the special example $\Sigma=(e_1,\,e_2,\,e_3)$ with $e_1=(1,2)$, $e_2=(1,3)$, and
$e_3=(2,3)$. Since the relevant observation times for this sequence are $t_1$ and $t_2$ separating the first two and the
second and third collisions, respectively, in this example we will consequently denote the velocities and space perturbations
observed at time $t_1$ with a superscript $-$, whereas the velocities and space perturbations
observed at time $t_2$ will be distinguished by a superscript $+$. (This is somewhat in contrast with the earlier notations, but here they
come rather handy.) 

The neutrality equations with respect to $e_1$ and $e_2$, along with the preservation of the center of mass are
\[
\alpha_1(v_1^--v_2^-)=\delta q_1^--\delta q_2^-,
\]
\[
\alpha_2(v_1^--v_3^-)=\delta q_1^--\delta q_3^-,
\]
\[
\delta q_1^-+\delta q_2^-+\delta q_3^-=0.
\]
From these equations we immediately get
\[
\delta q_1^-=\frac{1}{3}\alpha_1(v_1^--v_2^-)+\frac{1}{3}\alpha_2(v_1^--v_3^-),
\]
\[
\delta q_2^-=-\frac{2}{3}\alpha_1(v_1^--v_2^-)+\frac{1}{3}\alpha_2(v_1^--v_3^-),
\]
\[
\delta q_3^-=\frac{1}{3}\alpha_1(v_1^--v_2^-)-\frac{2}{3}\alpha_2(v_1^--v_3^-).
\]
By using the transformation equations through $e_2$ and the neutrality with respect to this collision
\[
\delta q_2^+=\delta q_2^-,
\]
\[
\delta q_1^-+\delta q_3^-=\delta q_1^++\delta q_3^+,
\]
\[
(\delta q_1^+-\delta q_3^+)-(\delta q_1^--\delta q_3^-)=\alpha_2\left[(v_1^+-v_3^+)-(v_1^--v_3^-)\right]
\]
one easily expresses the quantity $\delta q_2^+-\delta q_3^+$ as follows:
\[
\delta q_2^+-\delta q_3^+=-\alpha_1(v_1^--v_2^-)+\frac{1}{2}\alpha_2
\left[(v_1^--v_3^-)+(v_1^+-v_3^+)\right].
\]
Note that the linear coordinates $\alpha_1$ and $\alpha_2$ independently parametrize the
two-dimensional neutral space $\mathcal{N}_0(x;\,e_1,e_2)$.
From the last equation we see that the non-hyperbolicity $x\in J$ holds true precisely when
the vectors $v_1^--v_2^-$ and $(v_1^--v_3^-)+(v_1^+-v_3^+)$ are parellel. This parallelity condition 
defines a subspace $H$ for the vector $r(x)=(v_1^--v_3^-,\, v_1^+-v_3^+)$ with codimension $\nu-1$, which
codimension is $1$ exactly when $\nu=2$. (In the case $\nu\ge 3$ there is nothing to prove; the codimension
is already big enough.)

\medskip

The next result tells us that the collision $e_{n_0}$ decreases the dimension of the neutral space.

\begin{lm} \label{nzero-essential}
\[
\Delta(e_1,e_2,\dots,e_{n_0})<\Delta(e_1,e_2,\dots,e_{n_0-1}).
\]
\end{lm}

\begin{proof}
Proof by contradiction: assume that $\Delta(e_1,\dots,e_{n_0})=\Delta(e_1,\dots,e_{n_0-1})$. This assumption means that the actual CPF of (\ref{CPF})
(in which $m=n_0$) can be dropped from the whole system without affecting the solution set. Furthermore, by making the standard reduction
$\alpha_{n_0}=0$ for the advance $\alpha_{n_0}$ (which can be done by modifying the solution by adding to it a solution with all advances 
equal, and this chops off the dimension of the solution set by $1$) we can completely drop the $n_0$-th column from the system of
CPFs (\ref{CPF}). This shows that the two relative velocity components of $r(x)$ in (\ref{def-r}) have no effect on the solution set in question,
and this contradicts to the properties (i)--(ii) of the critical index $n_0$ listed in Definition \ref{n-zero}.
\end{proof}

The upcoming lemma tells us that the critical collision $e_{n_0}$ does not distinguish between the points of $J$ and
of $U_0\setminus J$.

\begin{lm} \label{no-distinguish}
\[
\Delta(e_1,e_2,\dots,e_{n_0})=\Delta_J(e_1,e_2,\dots,e_{n_0}).
\]
\end{lm}

\begin{proof}
Again a proof by contradiction: assume that 
$\Delta(e_1,\dots,e_{n_0})<\Delta_J(e_1,\dots,e_{n_0})$. According to Proposition \ref{DDP}, the neutral space
\[
\mathcal{N}_{n_0-1}\left(V_0(x);\, g_1(x),\dots,g_{n_0-1}(x)\right)
\]
is determined by the directions of the relative velocities $v^-_{i(l)}(x)-v^-_{j(l)}(x)$ and $v^+_{i(l)}(x)-v^+_{j(l)}(x)$
for $l=1,2,\dots,n_0-1$, whereas, according to (ii) of Definition \ref{n-zero}, these relative velocities do not determine whether
$x\in J$. On the other hand, the projection
\[
R_{n_0}\left[\mathcal{N}_{n_0-1}\left(V_0(x);\, g_1(x),\dots,g_{n_0-1}(x)\right)\right]
\]
of this neutral space onto $\delta q_{i(n_0)}-\delta q_{j(n_0)}$ determines if $x\in J$ is true or not. To see this we note
that, due to the assumption $\Delta(e_1,\dots,e_{n_0})<\Delta_J(e_1,\dots,e_{n_0})$, for the points $x\in U_0\setminus J$
the dimension of
\[
R_{n_0}\left[\mathcal{N}_{n_0-1}\left(V_0(x);\, g_1(x),g_2(x),\dots,g_{n_0-1}(x)\right)\right]
\] (which is 
\[
\text{dim}\left[\mathcal{N}_{n_0-1}\left(V_0(x);\, g_1(x),\dots,g_{n_0-1}(x)\right)\right]-
\text{dim}\left[\mathcal{N}_{n_0}\left(V_0(x);\, g_1(x),\dots,g_{n_0}(x)\right)\right]+1)
\]
is larger than the similar dimension for the points $x\in J$.
This, in turn, means that the directions of the relative velocities $v^-_{i(l)}(x)-v^-_{j(l)}(x)$ and $v^+_{i(l)}(x)-v^+_{j(l)}(x)$
($l=1,2,\dots,n_0-1$) determine if $x\in J$ is true or not, thus violating property (ii) of
$n_0$ listed in Definition \ref{n-zero}.
\end{proof}

\section{Finishing the proof of the Theorem} \label{finishing-theorem}

First we present the closing part of the proof by assuming that $\nu=2$. We remind the reader that the entire
proof of the Theorem is a proof by contradiction, so the coincidence (in a neighborhood $U_0$) of $J$ and the
past-singularity $K$ is assumed all along. Right after that we present the proof for the case $\nu\ge 3$,
which is just slightly more difficult technically than the case $\nu=2$. Thus, for now we assume that $\nu=2$.

Consider an arbitrary point $y_0\in J$. Let $\tau<0$ be the unique number such that
\begin{enumerate}
\item[$(1)$] $S^\tau y_0=y^* \in \mathcal{SR}^+_0$,
\item[$(2)$] $S^{(\tau,0)}y_0\cap \partial\mathbf{M}=\emptyset$.
\end{enumerate}
Here $\mathcal{SR}^+_0$ denotes the set of all singular reflections given with their outgoing
(post-singularity) velocity.

Select and fix a vector $w_0$, $w_0\perp v(y^*)$, such that
\begin{equation} \label{w-nought}
w_0\in\mathcal{N}_0\left(V_0(y^*);\, g_1(y^*),\dots,g_{n_0-1}(y^*)\right)\setminus
\mathcal{N}_0\left(V_0(y^*);\, g_1(y^*),\dots,g_{n_0}(y^*)\right).
\end{equation}
This is possible, due to lemmas \ref{nzero-essential}--\ref{no-distinguish}. Next we consider a smooth curve
$\gamma_0(s)$, $|s|<\vep_0$, $\gamma_0(0)=y^*$, $\gamma_0(s)\in\mathcal{SR}^+_0$, as follows:

\emph{Case A. If the singularity at $y^*$ is a double collision (a corner of the configuration space)}

\begin{enumerate}
\item[$(1)$] $v(\gamma_0(s))=\dfrac{v(y^*)+s\cdot w_0}{||v(y^*)+s\cdot w_0||}$,
\item[$(2)$] $q(\gamma_0(s))=q(\gamma_0(0))=q(y^*)$
\end{enumerate}
for $|s|<\vep_0$. 

\emph{Case B. If the singularity at $y^*$ is a tangency}

\begin{enumerate}
\item[$(1)$] $v(\gamma_0(s))=\dfrac{v(y^*)+s\cdot w_0}{||v(y^*)+s\cdot w_0||}$,
\item[$(2)$] $q(\gamma_0(s))=q(y^*)+\alpha\cdot w_0+\beta\cdot v(\gamma_0(s)))$
\end{enumerate}
($|s|<\vep_0$) so that the relation $\gamma_0(s)\in\mathcal{SR}^+_0$ still holds true. We note that the orders of
magnitude of the correction parameters $\alpha$ and $\beta$ are $\alpha=\text{O}(s^2)$, $\beta=\text{O}(s)$,
as a simple geometric observation shows.

Fix a time $t^*$, $t_{n_0-1}(y^*)<t^*<t_{n_0}(y^*)$, and investigate the image 
$S^{t^*}(\gamma_0(s))=\gamma^*(s)$ of the curve $\gamma_0$ under the $t^*$-iterate of the billiard flow. More precisely,
let us focus our attention on the projection
\begin{equation} \label{divergent}
\begin{aligned}
&\left(q_{i(n_0)}(\gamma^*(s))-q_{j(n_0)}(\gamma^*(s)),\, v_{i(n_0)}(\gamma^*(s))-v_{j(n_0)}(\gamma^*(s))\right) \\
&=(\ol{q}(s),\, \ol{v}(s))\in\mathbb{R}^2\times\mathbb{R}^2,
\end{aligned}
\end{equation}
and on the lines
\begin{equation} \label{lines}
L(s):=\left\{\ol{q}(s)+t\cdot \ol{v}(s)\big|\; t\in\mathbb{R}\right\}\subset\mathbb{R}^2.
\end{equation}

The following proposition directly follows from the definition (\ref{w-nought}) of $w_0$ and from the definition 
of the curve $\gamma_0\subset\mathcal{SR}^+_0$.

\begin{prp} \label{rotation}
The lines $L(s)$ rotate about a point $A$ of $\mathbb{R}^2$ in Case A, whereas they are tangential to a given ellipse
of $\mathbb{R}^2$ in Case B.
\end{prp}

\begin{rem} \label{degeneracy-warning}
We should note here that there is an exceptional subcase of Case B when the ellipse also degenerates to a point, just like in Case A.
This is the situation when the singularity at $y^*$ is a tangency but the projection $R_0(w_0)$ is parallel to the outgoing relative
velocity $v^+_{i(0)}-v^+_{j(0)}$ of the two particles $i(0)$ and $j(0)$ colliding tangentially at time zero. However, this degeneracy of the ellipse
does not cause any problem in the proof, for it is treated as the degeneracy in Case A.

We also note that in all of the cases above the directions of the lines $L(s)$ are properly changing at a non-zero rate, thanks to our
choice of $w_0$ with
\[
w_0\not\in\mathcal{N}_0\left(V_0(y^*);\, g_1(y^*),\dots,g_{n_0}(y^*)\right).
\]
\end{rem}

We remind the reader that, according to Proposition \ref{belongs-to-hyperplane}, the vectors
\[
r(\gamma_0(s))=\left(\ol{v}(s),\, \ol{v}^+(s)\right)
\]
belong to a given hyperplane $H(\gamma_0(0))=H(y^*)$ of $\mathbb{R}^4$ not depending on the parameter $s$. Here 
\begin{equation}\label{v-plus}
\ol{v}^+(s):=v^+_{i(n_0)}(\gamma_0(s))-v^+_{j(n_0)}(\gamma_0(s))
\end{equation}
denotes the outgoing $(i(n_0),j(n_0))$ relative velocity right after the collision $e_{n_0}=(i(n_0),j(n_0))$.
The reason why the hyperplanes $H(\gamma_0(s))$ are independent of $s$ is the following: Both the space and velocity
perturbations $q(\gamma_0(s))-q(\gamma_0(0))$ and $v(\gamma_0(s))-v(\gamma_0(0))$ belong to the neutral space 
$\mathcal{N}_0\left(V_0(y^*);\, g_1(y^*),\dots,g_{n_0-1}(y^*)\right)$ and, futhermore, they are proportional to each other.
One proves by a standard ``continuous induction'' that these properties remain true all the way until time $t^*$, thus
the (incoming and outgoing) relative velocities of the collisions $g_1, g_2,\dots,g_{n_0-1}$ are independent of the perturbation
parameter $s$.

\medskip

The proof of the Theorem will be complete as soon as we prove our

\begin{prp} \label{no-hyperplane}
Let $C_1\subset\mathbb{R}^2$ be an ellipse, possibly degenerated to a single point, $C_2\subset\mathbb{R}^2$
be a circle, so that none of $C_1$ or $C_2$ is lying inside the other one, i. e. they have at least two common tangent
lines. Suppose that $L(s)$, $|s|<\vep_0$, is a smooth family of oriented lines in $\mathbb{R}^2$ 
with the direction vector $\ol{v}(s)$ satisfying the following conditions:

\begin{enumerate}
\item[$(i)$] $L(s)$ is tangent to $C_1$ at the point of contact $A(s)$, and at $A(s)$ the direction vector $\ol{v}(s)$
agrees with a given orientation of $C_1$ (if $C_1$ is not a point),
\item[$(ii)$] $L(s)$ intersects $C_2$ in two points, out of which the one whose position vector makes the
smaller inner product with $\ol{v}(s)$ is denoted by $B(s)$,
\item[$(iii)$] $\dfrac{d}{ds}\alpha\left(\ol{v}(s)\right)>0$ for all $s$, $|s|<\vep_0$.
\end{enumerate}
Here $\alpha(\ol{v}(s))$ denotes the direction angle of the vector
$\ol{v}(s)$. Finally, let $\ol{v}^+(s)$ be the mirror image of $\ol{v}(s)$ under the orthogonal reflection
across the tangent line of the circle $C_2$ at the point $B(s)$.

We claim that there is no hyperplane $H\subset\mathbb{R}^2\times\mathbb{R}^2$ containing all the points
$(\ol{v}(s),\, \ol{v}^+(s))$ for $|s|<\vep_0$.
\end{prp}

\begin{proof}
A simple geometric inspection. We can assume, without restricting generality, that $\Vert\ol{v}(s)\Vert =1$.
We prove the proposition in the case when $C_1$ and $C_2$ have at least two common,
non-parallel tangent lines. The proof for the exceptional case, when this hypothesis is not satisfied, can be done with
some modifications, which we will show below right after completing the proof by using the hypothesis.

First of all, we can assume that the lines $L(s)$ depend on the parameter $s$
analytically. Then one can analytically extend the family of lines $L(s)$ to an interval of parameters
$I=[a,b]\supset (-\vep_0,\vep_0)$ by preserving all properties (i)--(iii) above so that $L(a)$ and $L(b)$ are non-parallel and
tangent to the circle $C_2$. If there was a hyperplane $H\subset\mathbb{R}^2\times\mathbb{R}^2$ containing all points
$(\ol{v}(s),\, \ol{v}^+(s))$ for $|s|<\vep_0$ then, by the reason of analyticity, the same containment 
$(\ol{v}(s),\, \ol{v}^+(s))\in H$ would be true for all $s$, $a\le s\le b$. Now we have that
\[
\begin{aligned}
& (\ol{v}(a),\, \ol{v}(a))\in H, \\
& (\ol{v}(b),\, \ol{v}(b))\in H,
\end{aligned}
\]
so $H$ contains the diagonal $\left\{(x,x)\big|\; x\in\mathbb{R}^2\right\}$ and, consequently,
the difference vectors $x-y$ for all $(x,y)\in H$ are parallel to each other.
But this is impossible, for the difference vectors $\ol{v}^+(s)-\ol{v}(s)$ can obviously rotate as $s$
varies in the parameter interval.

Finally, we show how to proceed in the case when $\ol{v}(a)$ and $\ol{v}(b)$ are parallel, i. e.
$\ol{v}(b)=-\ol{v}(a)$. We assume, contrary to the claim of the proposition, that there exists a
hyperplane $H\subset\BR^2\times\BR^2$ containing all vectors $(\ol{v}(s),\,\ol{v}^+(s))$, $s\in I$.
We take the limit 
\[
\lim_{s\to a^+} (s-a)^{-1/2}\left[(\ol{v}(s),\,\ol{v}^+(s))-(\ol{v}(a),\,\ol{v}^+(a))\right]
=(0,\,\xi)\in H,
\]
where $\xi\in\BR^2$, $\xi\ne 0$, $\xi\perp\ol{v}(a)$. This shows that for every $s\in I$ the vector
\[
\eta(s)=\ol{v}(s)-\langle\ol{v}^+(s), \ol{v}(a)\rangle\cdot\ol{v}(a)
\]
has the property that $(\eta(s), 0)\in H$. The vectors $\eta(s)$ ($s\in I$) must be mutually parallel, otherwise the
three-dimensional subspace $H$ of $\BR^2\times\BR^2$ would be equal to $\BR^2\times\langle\xi\rangle$, which would mean that
all outgoing vectors $\ol{v}^+(s)$ are parallel to $\xi$, but this is clearly not the case.

Denote the common line containing all the vectors $\eta(s)$ by $\mathcal{L}$. Clearly $\mathcal{L}$ is not parallel to
the vector $\ol{v}(a)$. We claim that $\mathcal{L}\perp\ol{v}(a)$. Indeed, the vectors $\ol{v}(s)$, $s\in I$, fill out one half of the
unit circle, thus in the case $\mathcal{L}\not\perp\ol{v}(a)$ there would be a parameter value $s$, $a<s<b$, such that
$\text{dist}(\ol{v}(s),\, \eta(s))>1$, which is impossible, for 
\[
\text{dist}(\ol{v}(s),\, \eta(s))=\left|\langle\ol{v}^+(s),\, \ol{v}(a)\rangle\right|\le 1.
\]
The fact $\mathcal{L}\perp\ol{v}(a)$, however, implies that 
$\langle\ol{v}^+(s)-\ol{v}(s),\, \ol{v}(a)\rangle=0$ for all $s\in I$, which is clearly a contradiction, since the nonzero
difference vectors $\ol{v}^+(s)-\ol{v}(s)$ are parallel to the rotating collision normal.
\end{proof}

\medskip

Finally, we complete the proof of the Theorem in the (somewhat more difficult) case $\nu\ge 3$, as follows.

We consider an arbitrary phase point $y_0\in J$, select the time $\tau<0$ and, correspondingly, the phase point
$y^*=S^\tau y_0$ just as before. Furthermore, the selection of a suitable tangent vector $w_0$ of (\ref{w-nought}), 
the construction of the smooth curve $\gamma_0(s)\in\mathcal{SR}_0^+$ ($|s|<\vep_0$), the selection of the separating
time $t^*$, the construction of the vectors
\[
(\ol{q}(s),\, \ol{v}(s))\in\mathbb{R}^\nu\times\mathbb{R}^\nu
\]
of (\ref{divergent}) and the construction of the lines
\[
L(s):=\left\{\ol{q}(s)+t\cdot \ol{v}(s)\big|\; t\in\mathbb{R}\right\}\subset\mathbb{R}^\nu
\]
of (\ref{lines}) are similar to what we did above in the case $\nu=2$, but now we have to exercise more care in the selection
of the neutral tangent vector $w_0$ of (\ref{w-nought}), see below.

Suppose, for a moment, that we have already chosen a suitable tangent vector
\[
w_0\in\mathcal{N}_0\left(V_0(y^*);\, g_1(y^*),\dots,g_{n_0-1}(y^*)\right)\setminus
\mathcal{N}_0\left(V_0(y^*);\, g_1(y^*),\dots,g_{n_0}(y^*)\right)
\]
of (\ref{w-nought}). 

The counterpart of Proposition \ref{rotation} is

\begin{prp} \label{counterpart_prop}
All the lines $L(s)$ ($|s|<\vep_0$) lie in the same two-dimensional affine subspace 
$\mathcal{P}=\mathcal{P}(y^*,w_0)$ of $\mathbb{R}^\nu$. These lines
rotate about a point $A$ of $\mathcal{P}$ in Case A, whereas they are tangential to a given
ellipse $C_1$ of $\mathcal{P}$ in Case B.
\end{prp}

\begin{rem} \label{degeneracy-warning-2}
We note that here Remark \ref{degeneracy-warning} again applies.
\end{rem}

Consider the smallest linear subspace $S=S(y^*,w_0)\subset\mathbb{R}^\nu$ of $\mathbb{R}^\nu$ containing the
affine plane $\mathcal{P}$. Clearly the dimension of $S$ is $3$ or $2$. By algebraic reasons there are two 
possibilities: Either the space $S=S(y^*,w_0)$ is $3$-dimensional for a typical pair $(y^*,w_0)$
($y^*=S^\tau y_0\in\mathcal{SR}_0^+$, $y_0\in U_0$, $w_0\in\mathcal{N}_0\left(V_0(y^*);\, g_1(y^*),\dots,g_{n_0-1}(y^*)\right)$,
$w_0\perp v(y_0)$), and in this situation we can assume that $\text{dim}S(y^*,w_0)=3$ {\it always} in our local analysis by choosing
a small enough open set $U_0$, or $\text{dim}S(y^*,w_0)=2$ for {\it all} such considered pairs. The next lemma shows that the latter
case is actually impossible.

\begin{lm}\label{high-dim}
It is not possible that $\text{dim}S(y^*,w_0)=2$ for every $y^*\in\mathcal{SR}_0^+$ ($y^*=S^\tau y_0$, $y_0\in U_0$) and for every
\[
w_0\in\mathcal{N}_0\left(V_0(y^*);\, g_1(y^*),\dots,g_{n_0-1}(y^*)\right)\setminus
\mathcal{N}_0\left(V_0(y^*);\, g_1(y^*),\dots,g_{n_0}(y^*)\right),
\]
$w_0\perp v(y^*)$.
\end{lm}

\begin{proof}
By way of contradiction, assume that $\text{dim}S(y^*,w_0)=2$ is always the case. This means that the velocities of the phase points
$y^*$ can be rotated along the curves $\gamma_0(s)\subset\mathcal{SR}_0^+$ in such a way that we obtain an $n_0$-th collision with a
collision normal vector parallel to the relative velocity of the colliding particles $i(n_0)$ and $j(n_0)$. (A so called ``head-on collision''.)
It is clear that the foliation of the manifold $\mathcal{SR}_0^+$ into the curves $\gamma_0(s)$ can be chosen to be smooth. Furthermore,
in order to reach a head-on collision from a given phase point $y^*\in\mathcal{SR}_0^+\cap U_0$ via the curve $\gamma_0(s)$ (with
$\gamma_0(0)=y^*$) it may be necessary to leave the small-sized local neighborhood $U_0$ in which we are working. During the
perturbation along the curve $\gamma_0(s)$ the times $t_k=t(e_k)$ of the collisions $e_k$ ($k=1,2,\dots, n_0-1$) also change, and this could change
the symbolic collision structure of the considered orbit segments. To avoid this problem, during the considered perturbations along the curves
$\gamma_0(s)$ we delete all hard core potentials of unduly arising new collisions, i. e. we allow two particles to freely overlap each other
if they would produce a collision not in the prescribed symbolic sequence $(e_1,e_2,\dots,e_{n_0-1})$. (A so called phantom dynamics.)

The above mean that the phase points $y^*\in\mathcal{SR}_0^+$ with head-on collisions $e_{n_0}$ form a codimension-one submanifold inside
$\mathcal{SR}_0^+$. However, this is impossible, since the singularity manifold $\mathcal{SR}_0^+$ can be smoothly foliated by convex, local
orthogonal manifolds, see \S4 in \cite{K-S-Sz(1990)}, and this shows that the codimension in $\mathcal{SR}_0^+$ of the set of phase points 
$y^*$ with a head-on collision $e_{n_0}$ is $\nu-1$, which is now at least $2$, a contradiction.
\end{proof}

Therefore, we may and we shall assume that the phase point $y^*=S^\tau y_0\in\mathcal{SR}_0^+$ is chosen 
(and fixed) in such a way that for the typical selection of $w_0$ in (\ref{w-nought}) it is true that
$\text{dim}S(y^*,w_0)=3$.

It is clear that the vector 
\[
r(\gamma_0(s))=\left(\ol{v}(s),\, \ol{v}^+(s)\right)
\]
varies in the $5$-dimensional linear subspace 
\[
\mathcal{P}'\times S\subset\mathbb{R}^\nu\times\mathbb{R}^\nu
\]
of $\mathbb{R}^{2\nu}$, where $\mathcal{P}'=\mathcal{P}'(y^*,w_0)$ is the $2$-dimensional linear subspace of $\mathbb{R}^\nu$
parallel to $\mathcal{P}$.

Let us focus on the hyperplane $H(\gamma_0(s))=H(y^*)$ of $\mathbb{R}^{2\nu}$, defined as
before. For the proof of the fact $H(\gamma_0(s))=H(\gamma_0(0))$ please see the paragraph containing
(\ref{v-plus}). The fact that the velocities $V_0(x),V_1(x),\dots,V_{n_0-1}(x)$ do not determine if the relation
$x\in J$ is true or not, has the following consequence.

\medskip

\begin{prp} \label{good_neutral}
For every singular phase point $y^*$ the neutral vector $w_0$ of \ref{w-nought} can be chosen in such a way that the hyperplane
$H(\gamma_0(s))=H(y^*)$ does not contain the subspace $\mathcal{P}'\times S$, i. e. 
$\text{dim}\left[(\mathcal{P}'\times S)\inter H(y^*)\right]=4$.
\end{prp}

\begin{rem}
The propery $\text{dim}\left[(\mathcal{P}'\times S)\inter H(y^*)\right]=4$ is an open property and the system 
in which it is defined is algebraic, so either this property holds on an open set with full measure inside the
singularity manifold (and then we can assume that it holds for every singular phase point in the local neighborhood 
$U_0$ that is chosen suitably small), or this property holds nowhere on the singularity manifold. 
In the indirect proof below we will assume the latter.
\end{rem}

\begin{proof}
A proof by contradiction. Assume that for every singular phase point $y^*$ (in $U_0$) and for every choice 
\[
w_0\in\mathcal{N}_0\left(V_0(y^*);\, g_1(y^*),\dots,g_{n_0-1}(y^*)\right)\setminus
\mathcal{N}_0\left(V_0(y^*);\, g_1(y^*),\dots,g_{n_0}(y^*)\right)
\]
the set containment
\[
\mathcal{P}'(w_0)\times S(w_0)\subset H=H(y^*)\subset\mathbb{R}^\nu \times \mathbb{R}^\nu
\]
is true. (The phase point $y^*=S^\tau y_0$ is now fixed.) This means that
\begin{equation}\label{P-part}
\bigcup_{w_0\in \mathcal{N}_0} \mathcal{P}'(w_0)\times\{0\}\subset H,
\end{equation}
and
\begin{equation}\label{S-part}
\bigcup_{w_0\in \mathcal{N}_0} \{0\}\times S(w_0)\subset H,
\end{equation}
where $\mathcal{N}_0=\mathcal{N}_0\left(V_0(y^*);\, g_1(y^*),\dots,g_{n_0-1}(y^*)\right)$.
Here is now the key observation: If we fix the manifold $\mathcal{W}_{n_0-1}(y^*)$, that is,
all the {\it directions} of all the relative velocities
$v^-_{i(k)}-v^-_{j(k)}$ and $v^+_{i(k)}-v^+_{j(k)}$ ($1\le k\le n_0-1$) for a phase point $y^*\in U_0$
and let all the other data vary then, according to Propositions \ref{DDP} and \ref{variations}, we also fix the neutral space
$\mathcal{N}_0=\mathcal{N}_0\left(V_0(y^*);\, g_1(y^*),\dots,g_{n_0-1}(y^*)\right)$, and at any time $t^*$
between $t_{n_0-1}$ and $t_{n_0}$ the data $\delta q$ and $\delta v$ vary in the neutral space
$\mathcal{N}_{n_0-1}\left(V_0(y^*);\, g_1(y^*),\dots,g_{n_0-1}(y^*)\right)$, which is also determined
by the manifold $\mathcal{W}_{n_0-1}(y^*)$, see again Proposition \ref{DDP}.
Therefore, the set containment relations (\ref{P-part})--(\ref{S-part}) mean that
\begin{equation}\label{everything-in-H}
\begin{aligned}
& R_{n_0}\left[\mathcal{N}_{n_0-1}\left(V_0(x);\, g_1(x),\dots,g_{n_0-1}(x)\right)\right] \\
& \times \text{span}\left\{q_{i(n_0)}-q_{j(n_0)},\, R_{n_0}\left[\mathcal{N}_{n_0-1}\left(V_0(x);\,
g_1(x),\dots,g_{n_0-1}(x)\right)\right]\right\}\subset H.
\end{aligned}
\end{equation}
for any phase point $x\in U_0\cap\mathcal{W}_{n_0-1}(y^*)$. We note here that not only the first factor of the
Cartesian product of (\ref{everything-in-H}) is constant on $\mathcal{W}_{n_0-1}(y^*)$, but the second one, as well.
The reason for this is that on $\mathcal{W}_{n_0-1}(y^*)$ the possible variations of the vector
$q_{i(n_0)}-q_{j(n_0)}$ belong to the space $R_{n_0}\left[\mathcal{N}_{n_0-1}\left(V_0(y^*);\, g_1(y^*),\dots,g_{n_0-1}(y^*)\right)\right]$,
see Proposition \ref{variations}.

The last set containment means that for any $x\in U_0\cap\mathcal{W}_{n_0-1}(y^*)$
it is true that $r(x)\in H(x)=H(y^*)$, so $x\in J$ for all such $x$, according to Proposition \ref{belongs-to-hyperplane}.
However, this contradicts to the fact that for the phase points $y\in U_0$ the manifolds $\mathcal{W}_{n_0-1}(y)$
are transversal to $J$, see Definition \ref{n-zero}.
\end{proof}

Our proof of the Theorem will be completed as soon as we prove the following counterpart of 
Proposition \ref{no-hyperplane}.

\begin{prp} \label{counterpart_geometry}
Let $\mathcal{P}'$ be a $2$-dimensional
linear subspace of the Euclidean space $\mathbb{R}^3$, $\mathcal{P}=\mathcal{P}'+x_0$ a coset of $\mathcal{P}'$
not containing $0$, $C_2\subset\mathbb{R}^3$ be the unit sphere of $\mathbb{R}^3$, $C_1\subset\mathcal{P}$ be an ellipse in $\mathcal{P}$,
possibly degenerated to a single point. Assume that the unit sphere $C_2$ intersects the affine plane $\mathcal{P}$ in a circle
$C$ and none of $C_1$ and $C$ lies completely inside the other one. Suppose that $L(s)$, $|s|<\vep_0$, is a smooth family of oriented 
lines in $\mathcal{P}$ with the direction vector $\ol{v}(s)$ satisfying the following conditions:

\begin{enumerate}
\item[$(i)$] $L(s)$ is tangent to $C_1$ at the point of contact $A(s)$, and at $A(s)$ the direction vector $\ol{v}(s)$
agrees with a given orientation of $C_1$ (if $C_1$ is not a point),
\item[$(ii)$] $L(s)$ intersects $C$ in two points, out of which the one whose position vector makes the
smaller inner product with $\ol{v}(s)$ is denoted by $B(s)$,
\item[$(iii)$] $\dfrac{d}{ds}\alpha\left(\ol{v}(s)\right)>0$ for all $s$, $|s|<\vep_0$.
\end{enumerate}
Here $\alpha(\ol{v}(s))$ denotes the direction angle of the vector $\ol{v}(s)$. Finally, let $\ol{v}^+(s)$ be the mirror 
image of $\ol{v}(s)$ under the orthogonal reflection across the tangent plane of the unit sphere $C_2$ at the point $B(s)$.

We claim that there is no ($4$-dimensional) hyperplane $H\subset \mathcal{P}'\times\mathbb{R}^3$ containing all the points
$(\ol{v}(s),\, \ol{v}^+(s))$ for $|s|<\vep_0$.
\end{prp}

\begin{rem}
In the proposition above the space $\BR^3$ plays the role of the space $S$ of Proposition \ref{good_neutral}.
\end{rem}

\begin{proof}
Very similar to the proof of Proposition \ref{no-hyperplane}. We assume again that $C_1$ and $C_2$ possess at least two
non-parallel tangent lines. (Otherwise the argument discussing the parallelity case $\ol{v}(b)=-\ol{v}(a)$ in the proof of
Proposition \ref{no-hyperplane} applies here with obvious modifications, which are left to the reader.) We can also assume that
the lines $L(s)$ depend on the parameter $s$ analytically, and this analytic family of lines $L(s)\subset\mathcal{P}$ is already extended 
to a parameter interval $I=[a,b]\supset(-\vep_0,\vep_0)$ by keeping the properties (i)---(iii) above, so that $L(a)$ and $L(b)$ are 
non-parallel and tangent to the circle $C=\mathcal{P}\cap C_2$. Suppose there is a hyperplane $H$ in the $5$-dimensional space 
$\mathcal{P}'\times\mathbb{R}^3$ containing all the points $(\ol{v}(s),\, \ol{v}^+(s))$ for $|s|<\vep_0$. By reasons of analyticity, the same 
membership relation $(\ol{v}(s),\, \ol{v}^+(s))\in H$ is true for all $s\in I$. The relations $(\ol{v}(a),\, \ol{v}(a))\in H$,
$(\ol{v}(b),\, \ol{v}(b))\in H$ imply that $H$ contains the diagonal $\left\{(v,v)\big|\; v\in \mathcal{P}'\right\}$,
which diagonal is the kernel of the linear map $\Psi:\; H\to \mathbb{R}^3$, $\Psi(v_1,v_2)=v_1-v_2$. Therefore, since
$\text{dim}H=4$ by our assumption, we get that $\text{dim}\Psi(H)\le 2$. However, for the points 
$(\ol{v}(s),\, \ol{v}^+(s))$, $a<s<b$, the lines spanned by the vectors $\ol{v}(s)-\ol{v}^+(s)$ fill out a (nonempty)
open part of a circular cone of $\mathbb{R}^3$, which cannot be the part of any subspace with dimension $\le 2$, so
the proposition and our non-coincidence theorem are now proved.
\end{proof}

\section{Proof of the Boltzmann-Sinai Ergodic Hypothesis \\ for all hard ball systems} \label{conclusion}

\begin{proof}
We carry out an induction on the number $N$ of elastically interacting
balls. For $N=2$ this is the classic result of Sinai and Chernov
\cite{S-Ch(1987)}. Suppose that $N>2$ and the result (ergodicity, the
Chernov-Sinai Ansatz, and complete hyperbolicity, implying the
Bernoulli mixing property, see \cite{C-H(1996)} and \cite{O-W(1998)})
has been proved for all systems of hard balls (of equal masses) on the
flat $\nu$-torus $\mathbb{T}^\nu$ with the number of balls less than $N$.
According to  Theorem 6.1 of \cite{Sim(1992)-I}, for almost every 
singular phase point $x\in\mathcal{SR}^+_0$ the forward orbit $S^{(0,\infty)}x$ of $x$

\medskip

\begin{enumerate}
\item[$(1)$] contains no singularity, and
\item[$(2)$] contains infinitely many connected collision graphs following each other in time.
\end{enumerate}

By Corollary 3.26 of \cite{Sim(2002)} such forward orbits $S^{(0,\infty)}x$ are sufficient (geometrically hyperbolic),
unless the phase point $x$ belongs to a countable family $J_1,\, J_2,\, \dots$ of exceptional, codimension-one, smooth, non-hyperbolicity
manifolds studied right here in this paper. By our Theorem, all these exceptional manifolds $J_k$ intersect $\mathcal{SR}^+_0$
in zero-measured subsets of $\mathcal{SR}^+_0$, and this proves the Chernov-Sinai Ansatz for our current system with $N$ balls.
Finally, the Theorem of \cite{Sim(2009)} gives us that the considered $N$-ball system is also ergodic, completely hyperbolic,
hence Bernoulli mixing.
\end{proof}

\bigskip

\noindent
Special thanks are due to N. I. Chernov and D. Dolgopyat for their illuminating questions and remarks.

\end{document}